\documentclass{ws-ijfcs}
\usepackage{enumerate}
\usepackage{url}
\usepackage{cite}
\begin{document}
	
	
	%
	%
	
	\title{Cayley automatic representations of 
		   wreath products}
	
	\author{Dmitry Berdinsky$^*$ \and Bakhadyr Khoussainov$^\dagger$}
	
	\address{Department of Computer, The University of Auckland \\
		     Private Bag 92019, Auckland, 1142, New Zealand\\
		\email{$^*$berdinsky@gmail.com} 
		\email{$^\dagger$bmk@cs.auckland.ac.nz}}
		
 	\maketitle
	
	
	\begin{abstract}
		We construct the representations of Cayley graphs of wreath 
		products using finite automata, pushdown automata and 
		nested stack automata. These representations are  
		in accordance with the notion of Cayley automatic groups
		introduced by Kharlampovich, Khoussainov and Miasnikov and
		its extensions introduced by Elder and Taback. 
		We obtain the upper and lower bounds for a length of
		an element of a wreath product in terms of the representations
		constructed.

	\end{abstract}
	
	\keywords{Finite automata; pushdown automata; nested stack automata;
		      Cayley graphs; wreath products.}
	
	\section{Introduction}	
	        
		 In this paper we study representations 
		 of Cayley graphs of wreath products 
		 of groups using finite automata, 
		 pushdown automata and nested stack automata. 
		 The representations considered in this paper are 
		 related to the notion of Cayley automatic
		 groups that was introduced  in \cite{KKM11}.  		 		 		 
		 The notion of Cayley automatic groups was 
		 introduced as a natural generalization 
		 of automatic groups in the sense of Thurston \cite{Epsteinbook}.

		 The set of Cayley automatic groups properly contains 
		 the set of automatic groups. 
		 In addition, the set of Cayley automatic groups includes finitely generated nilpotent groups of class at most two, the lamplighter group \cite{KKM11}, and the Baumslag--Solitar groups \cite{dlt14}. Cayley automatic groups retain some nice properties 
		 of automatic groups. They are closed under direct product, free product and finite extensions. The word problem in 
		 Cayley automatic groups is decidable in quadratic time.

		 We assume that the reader is familiar with the notions of finite
		 automaton and regular language.
		 Let $\Sigma$ be a finite alphabet.   
		 Put $\Sigma_\diamond = \Sigma \cup \{\diamond \}$, where 
		 $\diamond \notin \Sigma$. The convolution of 
		 two words $w_1, w_2 \in \Sigma^*$ is the string 
		 $w_1 \otimes w_2$		      
		 of length $\max\{|w_1|,|w_2|\}$ over the alphabet 
		 $\Sigma_\diamond \times \Sigma_\diamond$ defined as follows. 
		 The $k$th symbol
		 of the string is $(\sigma_1,\sigma_2 )$, 
		 where $\sigma_i$, $i=1,2$ is the $k$th symbol of 
		 $w_i$ if $k \leqslant |w_i|$ and
		 $\diamond$ otherwise. 
		 The convolution $\otimes R$  of a binary relation 
		 $R \subset \Sigma^* \times \Sigma^*$ is
		 $\otimes R = \{w_1 \otimes w_2 | (w_1, w_2) \in R\}$. 		
		 We say that a binary relation 
		 $R \subset \Sigma^* \times \Sigma^*$ is automatic if 
		 there exists a finite automaton in the alphabet 
		 $\Sigma_\diamond \times \Sigma_\diamond$ that accepts 
		 $\otimes R$. Such an automaton is called
		 two--tape synchronous finite automaton.

		 Let $G$ be a group generated by a  subset $S \subset G $. 
		 Consider the labeled and directed Cayley graph $\Gamma (G,S)$.
		 We say that $G$ is  Cayley automatic 
		 if some regular language $L\subset\Sigma^*$ over some alphabet $\Sigma$ uniquely represents elements of $G$ such that
		 for every  $s \in S$ the binary relation which is the set of directed edges of $\Gamma (G,S)$ labeled by $s$ is automatic. 
		 That is to say, the group $G$ is Cayley automatic if the labeled digraph $\Gamma (G,S)$ is an automatic structure in terms of \cite{KhoussainovNerode95}.

		 In \cite{ElderTabackCgraph}, Elder and Taback considered the extensions of the notion of Cayley automatic groups replacing the regular languages by more powerful languages.      
		 We denote by  $\mathcal{C}$  a class of languages; for example, it can be the class of regular languages, context--free languages or context--sensitive languages.    
		 The notion of Cayley automatic groups can be extended as follows. 
		 \begin{definition}	
		 \label{defCCayleyautomatic}   
		    Let $G$ be a group. We say that 
		    $G$ is $\mathcal{C}$ Cayley automatic if there exists 
		    a  subset $S \subset G$ generating $G$ for which the following properties hold: 
		    \begin{itemize}       
		 	   \item{There exists a bijection $\psi: L \rightarrow G$      
		 	   	     between a language $L \subset \Sigma^*$ 
		 	   	     from the class 
		 	   	     $\mathcal{C}$ and the group $G$;}
	           \item{For each $h \in S$ the language 
	           	$L_h = \{ w_1 \otimes w_2 | w_1, w_2 \in L, 
	           	\psi(w_1) h = \psi(w_2)\}$ is in the class 
	           	$\mathcal{C}$.}	 	
		 	\end{itemize}	
		 \end{definition}	
		 
		 In this paper all groups $G$ are finitely generated and generating sets $S$ are finite.
		
		 If $\mathcal{C}$ is the class of regular 
		 languages then $\mathcal{C}$ Cayley automatic 
		 groups are Cayley automatic groups.
		 The notion of Cayley automatic groups 
		 is invariant under the choice of generators.
		 Recall that the context--free and indexed languages are the ones 
		 that are recognizable by the pushdown automata and nested stack automata, respectively. 
		 Notice that for context--free and indexed 
		 Cayley automatic groups Definition \ref{defCCayleyautomatic} depends on the choice of generators.  
		 In terms of \cite{ElderTabackCgraph},
		 the group $G$, the subset $S$
		 and the finite alphabet $\Sigma$ in Definition \ref{defCCayleyautomatic} form a $\mathcal{C}$--graph 
		 automatic triple. 
		 
		 The known results on $\mathcal{C}$ Cayley automatic 
		 representations of groups are as follows. 
		 In \cite{KKM11}, Kharlampovich, Khoussainov and 
		 Miasnikov constructed the Cayley automatic 
		 representations for finitely generated nilpotent 
		 groups of class at most two.   
		 In \cite{dlt14}, we constructed the Cayley automatic 
		 representations for all Baumslag--Solitar groups. 
		 In \cite{ElderTabackCgraph}, Elder and Taback constructed
		 the deterministic non--blind 2--counter Cayley automatic
		 representation for the countably generated free group.
		 In \cite{ElderTabackF15}, they constructed 
		 the deterministic non--blind 
		 $1$--counter Cayley automatic representation 
		 for Thompson's group $F$.

		 In this paper we consider the cases when $\mathcal{C}$ 
		 is the class of regular languages, context--free languages 
		 and indexed languages. 		 
		 We assume that the reader is 
		 familiar with the notion of pushdown automata and 
		 context--free languages, and also nested stack automata 
		 and indexed languages. 
		 For the corresponding definitions we refer to 
		 \cite{Gilman05,Gilman98,HopcroftUllman}.

		 In this paper we  construct the Cayley 
		 automatic representation for 
		 wreath products $G \wr \mathbb{Z}$, 
		 the context--free Cayley automatic representation 
		 for wreath products $G \wr F_n$ and 
		 the indexed Cayley automatic representation for 
		 the wreath product $\mathbb{Z}_2 \wr \mathbb{Z}^2$.   
		 In each case we specify the set of generators for wreath products
		 with respect to which we consider their Cayley graphs. 
	     For the representations constructed we prove the inequalities of the form: 
		 \begin{equation}
		 \label{quasigeodesicourdefinition}    
		     \lambda |w| + \mu \leqslant |g|  \leqslant  \xi |w| + \delta, 	 
		 \end{equation}
		 where $|g|$ is the length of a group element $g$ 
		 with respect to chosen set of generators 
		 and $|w|$ is the length of the word $w$   
		 which is the representative of $g$, i.e., 
		 $\psi (w) = g$.

 		 In \cite{baumslagshapiro}, Baumslag, Shapiro and Short 
 		 introduced the notion of parallel poly--pushdown groups. 
 		 The definition of poly--pushdown groups uses 
 		 two--tape asynchronous automata instead of 
 		 synchronous ones used in Definition \ref{defCCayleyautomatic}.
 		 They have shown that the set of poly--pushdown groups is 
 		 closed under wreath product. 
 		 This implies that all wreath products  
 		 considered in this paper are parallel poly--pushdown groups.

 		 For wreath products there is an abundance of results on quantitative characteristics 
 		 such as growth rate \cite{Parry92}, 
 		 isoperimetric profiles \cite{Erschlerisoperimetric03} 
 		 and drift of simple random walks 
 		 \cite{dyubinaasymptotics2004,Dybina99Z2}.
 		 This makes studying representations of Cayley graphs of wreath products relevant to seeking connections between characteristics
 		 of groups and the computational power of automata that
 		 are sufficient to represent their Cayley graphs.
 		 In this paper we focus on the aforementioned classes of languages,
 		 i.e., regular, context--free and indexed.

         The  paper is organized as follows. 
         In Section \ref{preliminaries} we briefly recall 
         the definitions and notation for wreath products of groups. 
         In Section \ref{GwrZsection} we construct 
         the Cayley automatic representations for    
         wreath products $G \wr \mathbb{Z}$ and show 
         the inequalities of the form 
         \eqref{quasigeodesicourdefinition} for them. 
         In Section \ref{GwrFnsection} we construct  
         the context--free Cayley automatic representation for    
         wreath products $G \wr F_n$ and show 
         the inequalities of the form 
         \eqref{quasigeodesicourdefinition} for them. 
         In Section \ref{Z2wrZ2section} we construct 
         the indexed Cayley automatic representation for    
         the wreath product $\mathbb{Z}_2 \wr \mathbb{Z}^2$.
       		   
   \section{Wreath products of groups: definitions and notation}
   \label{preliminaries}	   
   	     
	     Recall the definition of the restricted wreath product 
	     $A \mathrel{\wr} B$. For more details on wreath products
	     see, e.g., \cite{KargapolovMerzljakov}.
         Given two groups $A$ and $B$,
         we denote by $A^{(B)}$ the set of all functions $B \rightarrow A$
         having finite supports.
         Recall that a function $f: B \rightarrow A$ 
         has finite support if $f (x) \neq e$
         for only finite number of $x \in B$, where $e$
         is the identity of $A$.
         Given $f \in A^{(B)}$ and $b \in B$, we define $f^b \in A^{(B)}$ as follows. Put $f^b (x) = f(bx)$ for all $x \in B$.
         The group $A \mathrel{\wr} B$ is the set product $B \times A^{(B)}$ with the group multiplication given by 
         $ (b, f)  \cdot (b', f') =   (b b' , f^{b'} f')$.

    For our purposes we use the converse order for representing the elements of a wreath product.
    Namely,  we represent an element of
    $A \mathrel{\wr} B$ as a pair $( f,   b )$, where $f \in A^{(B)}$ and $b \in B$.
    For such a representation, the group multiplication is given by
    $( f, b )  \cdot ( f', b' )    =  ( f f'^{\, b^{-1}},  b b'  )$.
    
    There exist natural embeddings $B \rightarrow A \mathrel{\wr} B$ and
    $A^{(B)} \rightarrow A \mathrel{\wr} B$ mapping $b$ to $ ({\bf e}, b )$
    and $f$ to $ (f, e)$ respectively, where
    ${\bf e}$ is the identity of $A^{(B)}$ and $e$ is the identity of $B$.
    For the sake of simplicity, we will identify  $B$ and $A^{(B)}$
    with the corresponding subgroups of $A \mathrel{\wr} B$.
    
    Recall that, according to \cite{baumslag61},
    the wreath product $A \wr B$ of two finitely presented 
    groups $A$ and $B$ is finitely presented iff either 
    $A$ is the trivial group or $B$ is finite. 
    Therefore, the wreath products 
    $G \wr \mathbb{Z}$, $G \wr F_n$ and $\mathbb{Z}_2 \wr \mathbb{Z}^2$
    considered here are not finitely presented. In particular, 
    these groups are not automatic \cite{Epsteinbook}.

    \section{The wreath products of groups with the  infinite cyclic group}
    \label{GwrZsection}

          We denote by $a$ the generator of 
          $\mathbb{Z} = \langle a \rangle$, 
          and by $h$ the nontrivial element of $\mathbb{Z}_2$. 
          We consider the Cayley graph of the lamplighter group 
          $ \mathbb{Z}_2 \wr \mathbb{Z}$ with respect to the generators 
          $a$ and $h$. 
          The automatic presentations for the Cayley graph of $\mathbb{Z}_2 \wr \mathbb{Z}$
          with respect to the generators $a$ and $h$ 
          were constructed in \cite{KKM11} and \cite{dlt14}.  
          In Theorem \ref{propL2cayleyautomatic} below we modify the  Cayley automatic representation used in~\cite{dlt14} (see Theorem~4).    
          This will enable us to get a simple proof of the inequalities
          of the form \eqref{quasigeodesicourdefinition}. 
         \begin{theorem}
         	\label{propL2cayleyautomatic}   
         	There exists a Cayley automatic representation
         	$\psi: L \rightarrow \mathbb{Z}_2 \wr \mathbb{Z}$ of the lamplighter group    	
         	 such that for every 
         	$g \in \mathbb{Z}_2 \wr \mathbb{Z}$ and its representative $w=\psi^{-1}(g)$ the 
         	inequalities 
         	$|w| - 1 \leqslant  |g| \leqslant 3 |w| - 2$    
         	hold.  
         \end{theorem}
         \begin{proof}
         	Recall that 
         	an element of $\mathbb{Z}_2 \wr \mathbb{Z}$ is a pair
         	$(f,z)$, where $f$ is a function 
         	$f: \mathbb{Z} \rightarrow \mathbb{Z}_2$ that has finite 
         	support and $z \in \mathbb{Z}$ is the position of the 
         	lamplighter.     	    	        	         	
         	In order to present  
         	$f(i) \in \mathbb{Z}_2$ we use the symbols $0$ and $1$: $0$ means
         	that a lamp in the position $z=i$ is unlit, i.e., $f(i)=e$; $1$ means that the lamp is lit, i.e., $f(i)=h$. 
         	By abuse of notation we will write $f(i)=0$ instead of
         	$f(i)=e$ and $f(i)=1$ instead of $f(i)=h$.         	
         	Recall that only 
         	a finite number of lamps are lit for each element of the group $\mathbb{Z}_2 \wr \mathbb{Z}$. 
         	To show the position of the origin $0 \in \mathbb{Z}$ we use
         	the symbols $A_0$ and $A_1$ if the lamp in the origin 
         	is unlit and lit respectively.    
         	To show the position of the lamplighter  
         	we use the symbols $C_0$ and $C_1$ if the lamp in 
         	the position of the lamplighter is unlit and lit respectively.  
         	In the case the lamplighter is at the origin  we use the symbols $B_0$ and $B_1$.  
         	
         	Given an element $(f,z) \in \mathbb{Z}_2 \wr \mathbb{Z}$, let $m$ be the 
         	smallest $i \in \mathbb{Z}$ such that 
         	$f(i) = 1$; if $f(i)= 0$ for all $i \in \mathbb{Z}$, 
         	then put $m  =0$. 
         	Put $\ell = \min \{m, z, 0 \}$.
         	Let $n$ be the largest 
         	$j$ such that $f(j) =1$; if $f(j)=0$ for all 
         	$j \in \mathbb{Z}$, then put $n = 0$.
         	Put $r = \max \{ n, z, 0\}$.   
         	Let us represent $(f,z)$ as follows: 
         	\begin{equation}
         	\label{code1L2}     
         	f(\ell)  f(\ell+1)  \dots f(-1)A_{f(0)} f(1) \dots
         	f(z-1) C_{f(z)} f(z+1) \dots f(r-1)f(r),
         	\end{equation}
         	here $z>0$ is assumed; if $z<0$, then $C_{f(z)}$ will appear on the left of $A_{f(0)}$.        	
         	In the case $z=0$ the word representing an element $(f,z)$ is 
         	\begin{equation} 
         	\label{code2L2}   
         	f(\ell) f(\ell+1)\dots f(-1) B_{f(0)} f(1) \dots f(r-1)f(r).
         	\end{equation}   
         	It can be observed that the language of the words representing all elements $(f,z)$ of the group $\mathbb{Z}_2 \wr \mathbb{Z}$ is regular.  
         	Let an element $g = (f,z)$ be represented by a word \eqref{code1L2}. Writing the words representing $g$ and $ga = (f,z+1)$   
         	one under another we have 
         	\begin{equation}
         	\label{rel1L2} 
         	\begin{array}{ccccccccccc}  
         	f(\ell) & \dots  & A_{f(0)}  & \dots & f (z-1) & C_{f(z)} & f(z+1) & \dots & f(r)  \\
         	f(\ell)  & \dots  & A_{f(0)}  & \dots & f (z-1) & f(z) & C_{f(z+1)} & \dots  & f(r)
         	\end{array}.
         	\end{equation}     
         	The other cases are considered similarly. 
         	From \eqref{rel1L2} 
         	it is clear that
         	the relation 
         	$\{ \langle g, ga \rangle | g \in \mathbb{Z}_2 \wr \mathbb{Z} \}$ is recognized by 
         	a synchronous two--tape finite automaton.            	         
         	Writing the words representing $g$ and $gh$ one under 
         	another we have
         	\begin{equation}
         	\label{rel3L2} 
         	\begin{array}{ccccccccccc}  
         	f(\ell) & \dots  & A_{f(0)}  & \dots & f (z-1) & C_{f(z)} & f(z+1) & \dots & f(r)  \\
         	f(\ell)  & \dots  & A_{f(0)}  & \dots & f (z-1) & C_{\overline{f(z)}} & f(z+1) & \dots  & f(r)
         	\end{array},
         	\end{equation}
         	where  $\overline{f(z)} = 1 - f(z)$.  
         	The other cases are considered similarly.            	
         	From \eqref{rel3L2} 
         	it is clear that
         	the relation 
         	$\{\langle g, gh \rangle | g \in 
         	\mathbb{Z}_2 \wr \mathbb{Z} \}$ is recognized by 
         	a synchronous two--tape finite automaton.

        Let us prove the inequalities $|w| - 1 \leqslant  |g| \leqslant 3 |w| - 2$.
        For the automatic representation constructed 
        the length of a word $w$
        representing an element $g= (f,z) $ is     		
        	\begin{equation}
        	\label{dlina1L2}
        	\begin{split}   
        	|w| =  |r - \ell | + 1 = 
        	|\max \{ n,z,0 \} - \min \{ m,z,0 \} | + 1 = \\
        	 \max \{ |n - m |, |n|, |m|, |n  - z|, |m - z|, |z| \} + 1. 
        	\end{split}   
        	\end{equation}
        For a given $f: \mathbb{Z} \rightarrow \mathbb{Z}_2$, we denote by 
        $\# \mathrm{supp} \, f $  the cardinality of the 
        set $\mathrm{supp} \, f =  \{ j \, | \, f(j) = h \}$.           
        	First we show that the word 
        	length of
        	$g $ with respect to 
        	the generators $a$ and $h$ is  
        	\begin{equation}
        	\label{lengthformula1L2}
        	\begin{split}
        	|g| =  \# \mathrm{supp} \, f  +   
        	\mathrm{min} \{ 2 \max\{-m,0\} + \max\{n,0\} + |z - \max\{n,0\}|, \\ 
        	2 \max\{n,0\} + \max \{-m, 0\} + |z + \max \{-m, 0\}| \}. 
        	\end{split}
        	\end{equation} 
          	By \cite{Taback03}, the left--first and the right--first normal forms of $g$ are 
        	\begin{equation*}
        	a_{i_1} \dots a_{i_p} a_{-j_1} \dots a_{-j_q} a^z,
        	\end{equation*}     
        	\begin{equation*}
        	a_{-j_1} \dots a_{-j_q} a_{i_1} \dots a_{i_p} a^z,  
        	\end{equation*}
        	where $i_p = n$ (if $n \geqslant 0$), $j_q =  -m$ (if $m \leqslant -1 $), 
        	$i_p > \dots > i_1 \geqslant 0$, $j_q > \dots > j_1 > 0$ and 
        	$a_{k} = a^k h a^{-k}$.  
        	It is proved \cite[Proposition~3.6]{Taback03}  that 
        	the word length of $g$ with respect to 
        	the generators $a$ and $h$ is
        	\begin{equation*}   
        	\label{length1L2}  
        	|g| =  p + q + \min\{2j_q + i_p + |z - i_p|, 2i_p + j_q + |z + j_q|\}. 
        	\end{equation*} 
        	Let us express $|g|$ in terms of $m \leqslant n$ for the three different cases:
        	\begin{itemize}
        		\item{ $m\leqslant -1$ and $n \geqslant 0$: 
        			$|g| = p + q + \min\{ -2m + n + |z - n|, 
        			2 n -  m + |z - m| \} $,}	 
        		\item{ $m \geqslant 0$: 
        			$ |g| = p + q +  n + |z - n| $, }
        		
        		\item{$n \leqslant -1$: 
        			$|g| = p + q -m + |z-m| $.}
        	\end{itemize} 	 
        	It can be seen that $ \# \mathrm{supp} \, f  = p + q$. 
        	Therefore, we obtain \eqref{lengthformula1L2}.

        	Let us prove the inequality $|g| \leqslant 3 |w| -2$. 
        	By \eqref{dlina1L2}, $|w| \geqslant n - m + 1$. 
        	Therefore, $|w| \geqslant \# \mathrm{supp}\,f$.    
        	Consider each of the three cases: $m \leqslant -1 <0 \leqslant n$, 
        	$n \leqslant -1$ and $0 \leqslant m $ separately.
        	\begin{itemize} 
        		\item{ The case $m \leqslant -1 <  0 \leqslant n$. 
        			If $z \geqslant n$,  
        			then we have: 
        			$-2m + n + |z-n| = -2m + z 
        			\leqslant 2(z - m)$. 
        			If $z \leqslant m$, then we have: 
        			$2n - m + |z - m| = 2n - m + m - z \leqslant 2 (n - z)$. 
        			If $m < z < n$, then we have: 
        			$-2m + n + |z- n|= 2(n-m) -z $ and 
        			$ 2n - m + |z -m|= 2(n-m) +z$. 
        			Therefore, by \eqref{dlina1L2}:  $\min\{ -2m + n + |z - n|, 2n - m + |z-m| \} 
        			\leqslant 2 (|w| - 1)$.   
        			Therefore, $|g| \leqslant 3 |w| - 2$.
        		}
        		\item{ The case $m \geqslant 0$. By \eqref{dlina1L2} we have: 
        			$ n  + |z - n| \leqslant 2(|w|-1)$.  
        			Thus, $|g| \leqslant  3 |w| -2$. 
        		}        
        		\item{ The case $n \leqslant -1$. By \eqref{dlina1L2} we have: 
        			$-m + |z - m| \leqslant 2 (|w| - 1)$.    
        			Therefore, $|g| \leqslant 3|w| -2$. 	
        		} 
        	\end{itemize}

        	Let us prove the inequality $|w| - 1 \leqslant |g|$. 
        	The identity $e \in  \mathbb{Z}_2 \wr \mathbb{Z}$ is represented by the word $B_0$. 
        	Therefore, the inequality holds for $g = e$. 
        	Suppose that the inequality holds for some $g \in \mathbb{Z}_2 \wr \mathbb{Z}$.  
        	For the Cayley automatic representation constructed the length of the word representing $gh$
        	equals $|w|$, and the lengths of the words representing $ga$
        	and $ga^{-1}$ are equal to either $|w|$, $|w|+1$ or $|w|-1$. 
        	This implies that the inequality holds for the elements $gh$, $ga$ and $ga^{-1}$. Thus, it holds for all $g \in \mathbb{Z}_2 \wr \mathbb{Z}$. 
        	It can be verified that both bounds $|w| - 1 \leqslant  |g| \leqslant 3 |w| - 2$    are  reached for an infinite number of elements $g \in \mathbb{Z}_2 \wr \mathbb{Z}$.              	                   
         \end{proof}
        
         Let $G$ be a Cayley  automatic group and         
         $\psi_G: L_G \rightarrow G$ be a Cayley automatic representation of $G$. We assume that the empty word $\varepsilon \notin L_G$.              	 
         Let $\{g_1, \dots , g_n\} \subset G$ be a finite set generating $G$.  
         We consider the Cayley graph of $G \wr \mathbb{Z}$ 
         with respect to the generators $g_1,\dots, g_n$ and $a$. 
          Put $d_j, j= 1,\dots,n$ to be the maximum number of the padding symbols $\diamond$ in the
          convolutions 
          $\psi_G ^{-1}(g) \otimes \psi_G ^{-1} (g g_j) $. 
          Put $K_0 = |\psi_G^{-1} (e)|$, where $e \in G$ is the 
          identity. 
          Put 
           $K = \max \{K_0, d_j \, | \, j \in [1,n]\}$.                 
         Theorem \ref{propL2cayleyautomatic} can be generalized, using essentially the same technique, to the following result.
     
         \begin{theorem}
         	\label{propGZautomatic} 	       	   
            There exists a Cayley automatic representation 
            $\psi: L \rightarrow G \wr \mathbb{Z}$ of the group $G \wr \mathbb{Z}$ such that for every 
            $g \in G \wr \mathbb{Z}$ and its representative $w = \psi^{-1}(g)$ the inequality 
            $\frac{1}{K} |w| - \frac{K_0}{K} \leqslant |g|$ holds.  
            Suppose that the inequality 
            $|g| \leqslant C |\psi_G^{-1}(g) |+ D$ 
            holds for all $g \in G$, 
            where $C>0$ and $D \geqslant 0$. 
            Then  for every 
            $g \in G \wr \mathbb{Z}$ and its representative 
            $w$ the  inequality
            $|g| \leqslant  (C + D + 2) |w| - 2$
             holds. 
         \end{theorem}
         \begin{proof}             	         	
         	For simplicity we may always assume that $L_G \subset \{0,1\}^*$ \cite{Blumensath99}.   	
         	We introduce 
         	two counterparts of the symbols $0$ and $1$:  $\underline{0}$ and 
         	 $\underline{1}$, respectively, which specify the beginning of a word. 
         	In order to specify the position of the origin $z = 0$ 
         	we use the symbols $A_0$ and $A_1$ depending on whether the word 
         	that represents the element of $G$ at $z=0$ has $0$ and $1$ as
         	the first letter. Similarly, we use the symbols $C_0$ and $C_1$ 
         	to specify the position of the lamplighter $z$. 
         	The symbols $B_0$ and $B_1$ are used if $z=0$.  
         	
         	Let us show two simple examples. Take an element $(f,1) \in G \wr \mathbb{Z}$   such that $f(j)=e$ for $j \notin [-1,2]$ and $f (-1) \neq e$, $f(0)$, $f(1)$ and  $f(2) \neq e$ are represented by the words $011$, $1001$, $01$ and $111$ respectively. Then  the element 
         	$(f,1)$ is represented by the word 
         	$
         	\underline{0}11A_1 001 C_0 1 \underline{1}11.      
         	$        
         	Take an element $(f,0) \in G \wr \mathbb{Z}$ such that $f(j)=e$ 
         	for $j \notin [-1,1]$ and $f(-1) \neq e$, $f(0)$ and $f(1) \neq e$ are 
         	represented by the words $111$, $000$ and $01$ respectively. 
         	Then the element $(f,0)$ is represented 
         	by the word 
         	$
         	\underline{1}11 B_0 00 \underline{0}1.
         	$  
         	
         	By abuse of notation, we denote by $f(j)$ the word  representing the group element $f(j) \in G$ for which the first letter $\sigma$ is changed to the underlined one $\underline{\sigma}$. 
         	We denote by $A_{f(0)}$, $B_{f(0)}$ and $C_{f(z)}$ the corresponding words for which the first letter 
         	$\sigma$ is changed to $A_\sigma$, $B_\sigma$ and 
         	$C_\sigma$, respectively.

         	Let  $(f,z) \in G \wr \mathbb{Z}$.          	
         	The numbers $\ell$ and $r$ have the same meaning as in the proof of Theorem \ref{propL2cayleyautomatic}. Similar to \eqref{code1L2} and 
         	\eqref{code2L2} an element $(f,z)$ is represented by the word        	
         	$f(\ell) \dots A_{f(0)} \dots C_{f(z)} \dots f (r)$
         	and an element $(f,0)$ is represented by the word
         	$f(\ell) \dots B_{f(0)} \dots f(r)$.
         	It is clear that the relation 
         	$\{\langle g,ga\rangle\,|\,g \in G\wr\mathbb{Z}\}$
         	is recognizable by a two--tape synchronous finite automaton.
         	Since $G$ is Cayley automatic, 
         	for every $j =1,\dots,n$ the relation $\{\langle g, g g_j  \rangle \, | \, g \in G \wr \mathbb{Z} \}$
         	is recognizable by a two--tape synchronous finite automaton.
         	The inequalities  $\frac{1}{K} |w| - \frac{K_0}{K} \leqslant |g|$ and $|g| \leqslant  (C + D + 2) |w| - 2$
         	can be obtained using the same technique as in 
         	Theorem \ref{propL2cayleyautomatic}.
         	
      \end{proof}  
     
      \begin{remark}
      	 Suppose that the representation $\psi_G : L_G \rightarrow G$
      	 is Cayley biautomatic (see \cite{KKM11} for the definition of Cayley biautomatic
      	 groups.). It can be verified that then  
      	 the representation 
      	 $\psi : L \rightarrow G \wr \mathbb{Z}$ constructed in Theorem \ref{propGZautomatic} is Cayley biautomatic.  	  
      \end{remark}

    \section{The wreath products of groups with a free group}
    \label{GwrFnsection}
          
          We denote by $a$ and $b$ the generators of the free group $F_2= \langle a, b \rangle$, 
          and by $h$ the nontrivial element of $\mathbb{Z}_2$. 
          We consider the Cayley graph of the            
          wreath product $\mathbb{Z}_2 \wr F_2$ with 
          respect to the generators $a,b$ and $h$. 
          Recall that an element of $\mathbb{Z}_2 \wr F_2$ 
          is a pair $(f,z)$, where $f$ is a function $f: F_2 \rightarrow \mathbb{Z}_2$ that has 
          finite support and $z \in F_2$ is the position of the lamplighter.
          We have the following theorem.               
         \begin{theorem}
         	\label{Z2wrF2automatic}   
         	There exists a context--free Cayley automatic representation 
         	$\psi : L \rightarrow \mathbb{Z}_2 \wr F_2$ 
         	of the group $\mathbb{Z}_2 \wr F_2$
         	such that for every $g \in \mathbb{Z}_2 \wr F_2$
         	and its representative $w = \psi^{-1} (g)$ the 
         	inequalities 
         	$
         	\frac{1}{3} |w| - \frac{1}{3}  \leqslant 
         	|g| \leqslant  3 |w| - 2
         	$  
         	hold.	   
         \end{theorem}	
         \begin{proof}          		         	          	          	
           	In order to construct a context--free Cayley automatic representation of 
           	$\mathbb{Z}_2 \wr F_2$  
         	we extend the Cayley automatic representation of $\mathbb{Z}_2 \wr \mathbb{Z}$ obtained in Theorem \ref{propL2cayleyautomatic}.      
            In the Cayley automatic representation of 
            $\mathbb{Z}_2 \wr \mathbb{Z}$ we use the symbols 
            $0,1,A_0,A_1,B_0,B_1,C_0,C_1$.
            In the context--free Cayley automatic representation
            of $\mathbb{Z}_2 \wr F_2$ we use the brackets 
            $($, $)$ and $[$,  $]$.            
         	Along with the symbols $0$ and $1$ we use $D_0,E_0$ and $D_1,E_1$. Along with 
         	the symbols $A_0,A_1,B_0,B_1,C_0,C_1$ 
         	we use  $D_0 ^A, D_1 ^A, D_0 ^B, D_1 ^B, D_0 ^C,D_1 ^C,
         	E_0 ^C, E_1 ^C$.
         	The symbols $A_0,A_1,D_0 ^A, D_1 ^A$ are used to 
         	show the position of the origin $e\in F_2$. 
            The symbols $C_0,C_1,D_0 ^C, D_1 ^C, E_0 ^C, E_1 ^C$ are used to 
            show the position of the lamplighter $z\in F_2$.  	        	
         	The symbols $B_0, B_1, D_0 ^B, D_1 ^B$ 
         	are used if the lamplighter is at the origin. See also the
         	meaning of the symbols $A_0, A_1, B_0 , B_1, C_0, C_1$ in 
         	Theorem \ref{propL2cayleyautomatic}.


         	 We say that a symbol is an $A$--, $B$--, $C$--, $D$-- and $E$--symbol if it belongs to the set 
         	 $\{A_0,A_1, D_0 ^A, D_1 ^A \}$, 
         	 $\{B_0, B_1, D_0 ^B, D_1 ^B \}$, 
         	 $\{ C_0, C_1, D_0 ^C, D_1 ^C,E_0 ^C, E_1 ^C\}$,
         	 $\{D_0,D_1,D_0 ^A, D_1 ^A, D_0 ^B, D_1 ^B, D_0 ^C, D_1 ^C\}$
         	 and $\{E_0,E_1, E_0 ^C, E_1 ^C \}$,
         	 respectively. We say that a symbol is basic if it belongs 
         	 to the set $\{0,1,A_0,A_1,B_0,B_1,C_0,C_1\}$.  
                
            For a given $s \in F_2$, denote by 
            $r(s) \in \{a,a^{-1},b, b^{-1}\}^*$ the reduced word
            representing $s$. We denote by $F_a$ the set
            of all group elements $s \in F_2 $ for which  
            $r(s) = a w$ or $r(s) =a^{-1}w$, $w \in \{a,a^{-1},b, b^{-1}\}^*$. We denote by $F_b$ the set
            of all group elements $s \in F_2 $ for which  
            $r(s) = b w$ or $r(s) =b^{-1}w$, $w \in \{a,a^{-1},b, b^{-1}\}^*$. It is clear that $F_2 = F_a \cup F_b \cup \{e\}$.   
            We denote by $H$ the subgroup of $\mathbb{Z}_2 \wr F_2$ 
            generated by $a$ and $h$. It is clear that $H$ is isomorphic
            to $\mathbb{Z}_2 \wr \mathbb{Z}$.

            For a given $(f,z) \in \mathbb{Z}_2 \wr \mathbb{Z}$, depict it in a way shown Fig.~\ref{freegraph1}. 
            Let us consider the horizontal line going through the identity $e \in F_2$. For a vertex $s\in F_2$ on this line,
            put $V_s = \{p \in F_b \, | \,  f (s p) = h \vee z= sp \}$.  
            Scan this line from the left to the right. 
            If $V_s = \varnothing$, then write the corresponding
            basic symbol (see Theorem \ref{propL2cayleyautomatic}).
            If $V_s \neq \varnothing$, then write the corresponding  $D$--symbol. 
            For the element in Fig.~\ref{freegraph1} (left) we get 
            $11 D_0 ^A D_0 1$. 
            For the element in Fig.~\ref{freegraph1} (right) we get 
            $D_0 A_1 D_0 1$.          
            If $(f,z) \in H$, then we obtain  
            the same representative as in Theorem \ref{propL2cayleyautomatic}. If $(f,z) \notin H$, then
             $D$--symbols occur.             
             In this case we continue as follows. 
            
            Take any occurence of $D$--symbol.  
            This occurence corresponds to some vertex $s \in F_2$. Let us consider a vertical line going through
            $s$. For a vertex $t \in F_2$, $t \neq s$ on this line, 
            put $H_t = \{ p \in F_a \, | \, f(tp) = h  \vee z = tp\}$.                
            Insert the brackets $($ and $)$ around the occurence of a $D$--symbol.
            Scan, omitting $s$, this line from the bottom to the top. 
            If $H_t = \varnothing$, then write the corresponding 
            basic symbol inside the brackets. 
            If $H_t \neq \varnothing$, then write the corresponding 
            $E$--symbol inside the brackets.                         
            Do it for every occurence of a $D$--symbol. 
            For the element in Fig.~\ref{freegraph1} (left) we get 
            $11 (1E_0 D_0^A E_0) (E_0 D_0 E_1) 1 $. 
            For the element in Fig.~\ref{freegraph1} (right) we get 
            $(D_01)A_1(E_0D_0E_1 ^C)1$.
            If no $E$--symbols occur then we stop.  
            If $E$--symbols occur, we insert the brackets $[$ 
            and $]$ around each occurence and repeat the step above
            for horizontal lines. 
            
            We continue this procedure until no new $D$-- or 
            $E$--symbols occur. After the procedure is finished, 
            the result is the representative $w = \psi^{-1} (g)$ of 
            $g = (f,z)$.        		 	
         	    For the element in Fig.~\ref{freegraph1} (left) the procedure of  constructing the representative is 
         	    $11 D_0 ^A D_0 1 \rightarrow$ $11 (1E_0 D_0^A E_0) (E_0 D_0 E_1) 1 \rightarrow$$11 (1[1E_01] D_0 ^A [E_0 D_1]) 
         	    ([1E_0] D_0  [1E_1])1 \rightarrow$ $11 (1[1E_01] D_0 ^A [E_0 (C_1 D_1)])([1E_0] D_0 [1E_1])1$. 
         		For the element in Fig.~\ref{freegraph1} (right) it is 
         		$D_0 A_1 D_0 1 \rightarrow$ $(D_01)A_1(E_0D_0E_1 ^C)1\rightarrow$ $(D_01)A_1([1E_0]D_0[1E_1 ^C])1$.

         		\begin{figure}[b]
         			\centerline
         			{\includegraphics[height=5.6cm]{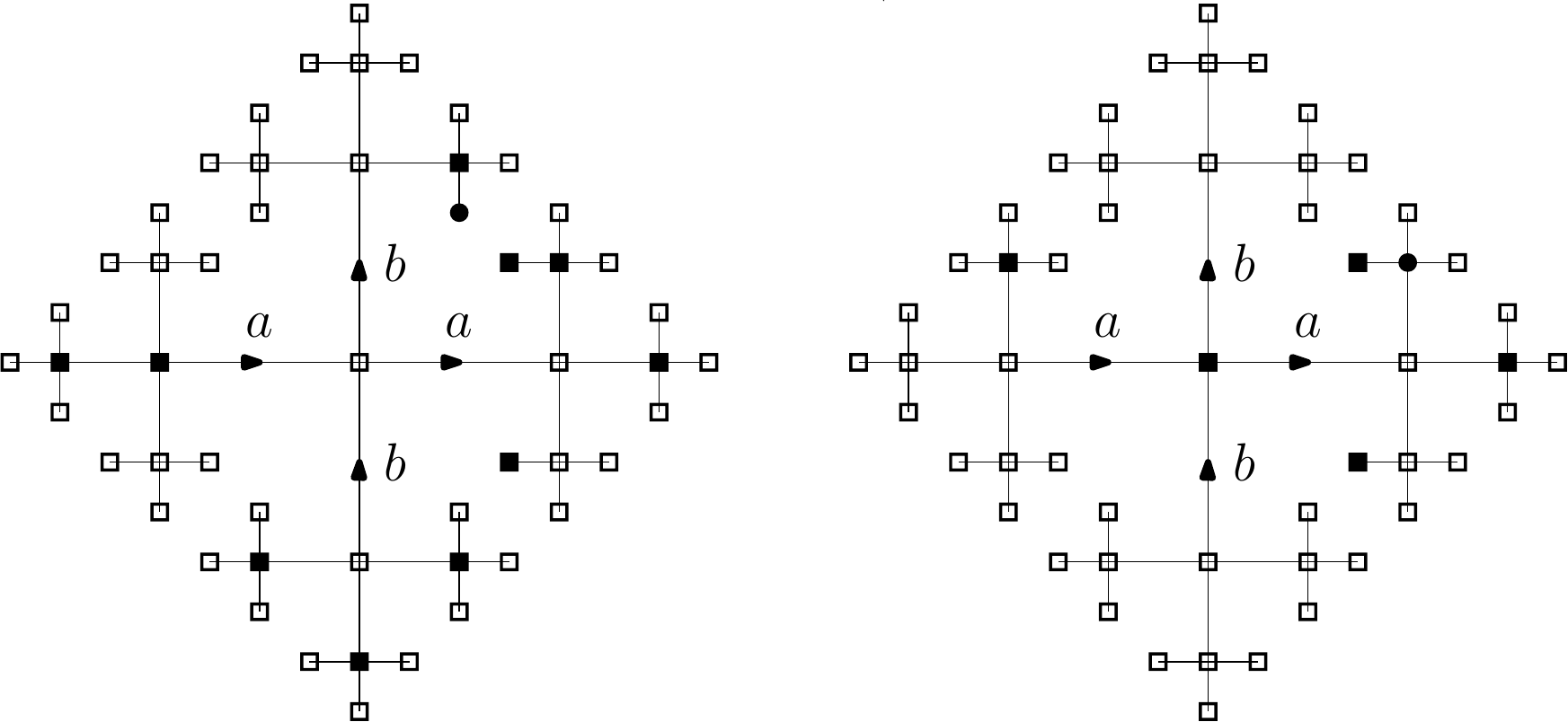}}
         			\vspace*{8pt}
         			\caption{A white box means that the value of a function
         				$f: F_2 \rightarrow \mathbb{Z}_2$ is $e \in \mathbb{Z}_2$, a black box means that it is $h \in \mathbb{Z}_2$,	a black disk specifies the position of the lamplighter $z$  and tells us that the value of $f$ is $h$. For the element to the left the word representing it is
         				$11(1[1E_01]D_0 ^A[E_0(C_1D_1)])([1E_0]D_0[1E_1])1$, 
         				for the element to the right it is 
         				$(D_01)A_1 ([1E_0]D_0[1E_1^C])1$.
         			}        		          	                     	 \label{freegraph1}
         		\end{figure}

           	Put $\Sigma = \{0,1,D_0,D_1,E_0,E_1, 
           	 (,),[,],A_0,A_1,B_0,B_1,C_0,C_1,D_0 ^A, D_1 ^A, D_0 ^B,D_1 ^B,$ $D_0 ^C,D_1 ^C,E_0 ^C$, $E_1 ^C\}$.	
         	Let $L \subset \Sigma^*$ be the language of representatives $w$ of all elements $g \in \mathbb{Z}_2 \wr F_2$.
            It can be seen that the language $L$ consists of the words 
            satisfying the following properties.
         		\begin{itemize}     
         			\item{The configuration of brackets $($, $)$, $[$, $]$ is balanced and, moreover, generated by the context--free grammar 
         				$S \to SS \, | \, (T) \, | \, \varepsilon, \,
         				T \to TT \, | \, [S] \,  | \, \varepsilon$
         				with the axiom $S$.}

         			\item{Each pair of matched brackets $($ and $)$  
         				is associated with a $D$--symbol 
         				which is placed inside these brackets, but not inside any other pair of matched brackets between them. That is, 
         			    the configuration of the subword between any 
         			    two matched brackets $($ and $)$ is 
         				$(p[\dots]q \dots r[\dots]s \,\sigma \,t[\dots]u\dots v[\dots]w)$,
         				where $\sigma$ is a $D$--symbol and 
         				$p,q,r,s,t,u,v,w \in \{0,1,C_0,C_1\}^*$. 
         				} 
         			\item{The $D$--symbols $D_0 ^A, D_1 ^A, D_0 ^B, D_1 ^B$ are allowed to be associated only with a matched pair of brackets $($ and $)$ of the first level.}
         			\item{Each pair of matched brackets 
         				$[$ and $]$ is associated with 
         			  	an $E$--symbol 
         				which is placed inside these brackets 
         				but not inside any other pair 
         				of matched brackets between them.
         				That is, 
         				the configuration of the subword between any 
         				two matched brackets $[$ and $]$ is 
         				$[p(\dots)q \dots r(\dots)s \,\sigma \,t(\dots)u\dots v(\dots)w]$,
         				where $\sigma$ is an $E$--symbol and 
         				$p,q,r,s,t,u,v,w \in \{0,1,C_0,C_1\}^*$. 
         				}
         			   
         			\item{Each pair of matched brackets is separated by at least two symbols.}
         			\item{The subwords $(0$, $0)$, $[0$ and $0]$ 
         				are not allowed.}
         			\item{The symbol $0$ is not allowed to be 
         				  the first or the last one of a word.}	   	
         			\item{Among the symbols $A_0$,$A_1$,$B_0$,$B_1$,$C_0$,$C_1$,$D_0^A$,$D_1^A$,$D_0^B$,$D_1^B$,$D_0 ^C$,$D_1 ^C$,$E_0^C$,$E_1^C$ each word of $L$ contains either exactly one occurrence of a $B$--symbol and no $A$--symbols and $C$--symbols, or  exactly one occurrence of an $A$--symbol and of a $C$--symbol, and no $B$--symbols.} 
         		        			
         		\end{itemize}

         		 %

         		It can be seen that $L$ is recognizable by a deterministic pushdown automaton. 
         		The right multiplication by $h$ either interchanges 
         		$C_0$ and $C_1$, $D_0 ^B$ and $D_1 ^B$, $D_0 ^C$ and $D_1 ^C$, or $E_0 ^C$ and $E_1 ^C$.
         		Therefore, the language 
         		$L_h = \{u\otimes v  |  u, v \in L,  \psi (v) =  \psi (u) h  \}$ is clearly context--free.        
         		The right multiplication by $a$ (or, $b$) moves the 
         		lamplighter by one step to the right (or, up). It is can be verified that  
         		the languages 
         		$L_a = \{u \otimes v | u,v \in L, \psi (v) = \psi(u) a \}$ 
         		and 
         		$L_b = \{u \otimes v | u,v \in L, \psi (v) = \psi(u) b \}$ 
         		are context--free as well.             
         	
         	Let us prove the inequality 
         	$\frac{1}{3} |w| - \frac{1}{3}  \leqslant  |g|$. 
         	It is equivalent to 
         	$|w| \leqslant 3 |g| +1$.
            The representative of 
           	$e \in \mathbb{Z}_2 \wr F_2$ is the word $A_0$ 
           	of length $1$, so the inequality  holds for $g=e$. 
           	For every $g \in \mathbb{Z}_2 \wr F_2$ 
           	the lengths of the representatives for $g$ and $gh$ are the same. It can be seen that 
           	for every $g \in \mathbb{Z}_2 \wr F_2$ the lengths 
           	of the representatives for $g$ and $ga$ (or $gb$)     
           	differ by at most $3$. 
           	Therefore, the inequality $|w| \leqslant 3 |g| + 1$
           	holds for all $g \in \mathbb{Z}_2 \wr F_2$.

           	Let us prove by induction the inequality 
           	$|g| \leqslant  3 |w| - 2$.
           	For the elements of the subgroup $H$ 
           	the representatives are the same as in 
           	Theorem \ref{propL2cayleyautomatic}, so
            the inequality $|g| \leqslant  3 |w| - 2$ is satisfied 
            for them.           	
           	For the inductive step, we observe that a 
           	word $w$ that represents any element 
           	$g \in \mathbb{Z}_2 \wr F_2$ has the form 
           	$v_0 (w_1) v_1 (w_2) v_2 \dots v_{n-1} (w_n) v_n$, 
           	where the words $v_0, v_1, \dots, v_n$ do not 
           	contain brackets. 
           	We 
           	obtain that $|g| \leqslant \sum_{i=1}^n (3 |w_i| - 2) 
           	+ 3 (|v_0 \dots v_n| + n) - 2 \leqslant 3 |w| - 2$. Therefore, 
           	the inequality $|g| \leqslant  3 |w| - 2$
           	holds for all $g \in \mathbb{Z}_2 \wr F_2$. 
            It can be verified that both bounds 
            $ \frac{1}{3} |w| - \frac{1}{3}  \leqslant  |g| \leqslant  3 |w| - 2 $ are reached for an infinite number of elements $g \in \mathbb{Z}_2 \wr F_2$.
        \end{proof}  
       

     Theorem \ref{Z2wrF2automatic} can be generalized, using essentially the same technique, to the following result.        
     \begin{theorem}
     \label{Z2wrFnautomatic}   
      There exists a context--free Cayley automatic representation 
      $\psi : L \rightarrow \mathbb{Z}_2 \wr F_n$ 
      of the group $\mathbb{Z}_2 \wr F_n$, $n \geqslant 2$
      such that for every $g \in \mathbb{Z}_2 \wr F_n$
      and its representative $w = \psi^{-1} (g)$ the 
      inequalities 
      $
      \frac{1}{2n-1} |w| - \frac{1}{2n-1}  
      \leqslant |g| \leqslant 3 |w| - 2
      $  
      hold.	   
  \end{theorem}	  
        
      Let $\psi_G: L_G \rightarrow G$ be a context--free Cayley 
      automatic representation of a group $G$
      with respect to a set of generators 
      $\{g_1, \dots, g_m\} \subset G$
      such that the following holds for some constant $N$:   
      for all $u,v \in L_G$ such that $\psi_G (u) g_j = \psi_G (v)$, 
      $j = 1,\dots,m$ the inequality  
      $||u|-|v|| \leqslant N$ holds.       
      We assume that the empty word $\varepsilon \notin L_G$.  
      Let $a_1,\dots,a_n$ be the generators of the free group $F_n$. 
      We consider the Cayley graph of $G \wr F_n$ 
      with respect to the generators $g_1,\dots, g_m,a_1,\dots,a_n$.
      Put $K_0 = |\psi_G^{-1} (e)|$, where $e \in G$ is the 
      identity.
      Put $K = \max\{ K_0 + 2 (n-1), N \}$.
      Theorem \ref{Z2wrFnautomatic} can be generalized to the following result  (cf. Theorem \ref{propGZautomatic}).

      \begin{theorem}
      \label{GwrFncontextfreeautomaticineq}    	       	   
        	There exists a context--free Cayley automatic representation 
        	$\psi: L \rightarrow G \wr F_n$ of the group $G \wr F_n$ such that for every 
        	$g \in G \wr F_n$ and its representative $w = \psi^{-1}(g)$ the inequality 
        	$\frac{1}{K} |w| - \frac{K_0}{K} \leqslant |g|$ holds.  
        	Suppose that the inequality 
        	$|g| \leqslant C |\psi_G^{-1}(g) |+ D$ 
        	holds for all $g \in G$, 
        	where $C>0$ and $D \geqslant 0$. 
        	Then  for every 
        	$g \in G \wr F_n$ and its representative 
        	$w$ the  inequality
        	$|g| \leqslant  (C + D + 2) |w| - 2$
        	holds.                  
      \end{theorem}

    \section{The wreath product $\mathbb{Z}_2 \wr \mathbb{Z}^2$}
    \label{Z2wrZ2section}    
       We denote by $x$ and $y$ the standard generators of $\mathbb{Z}^2$, and 
       by $h$ the nontrivial element of $\mathbb{Z}_2$.    
       Let us consider the Cayley graph of the wreath product $\mathbb{Z}_2 \wr \mathbb{Z}^2$ with respect to the generators $x,y$ and $h$.
       Recall that 
       an element of $\mathbb{Z}_2 \wr \mathbb{Z}^2$ is a pair
       $(f,z)$, where $f$ is a function 
       $f: \mathbb{Z}^2 \rightarrow \mathbb{Z}_2$ that has finite 
       support and $z \in \mathbb{Z}^2$ is the position of the 
       lamplighter. 
       We have the following theorem. 
       \begin{theorem}
       	\label{Z2wrZ2indexautomaticprop}   
       	There exists an indexed Cayley automatic representation 
       	$\psi: L \rightarrow \mathbb{Z}_2 \wr \mathbb{Z}^2$
       	for the wreath product $\mathbb{Z}_2 \wr \mathbb{Z}^2$ 
       	such that $L$ is a regular language. 
       \end{theorem} 
       \begin{proof}          
       	\begin{figure}[b]
       		\centerline
       		{\includegraphics[height=6.0cm]{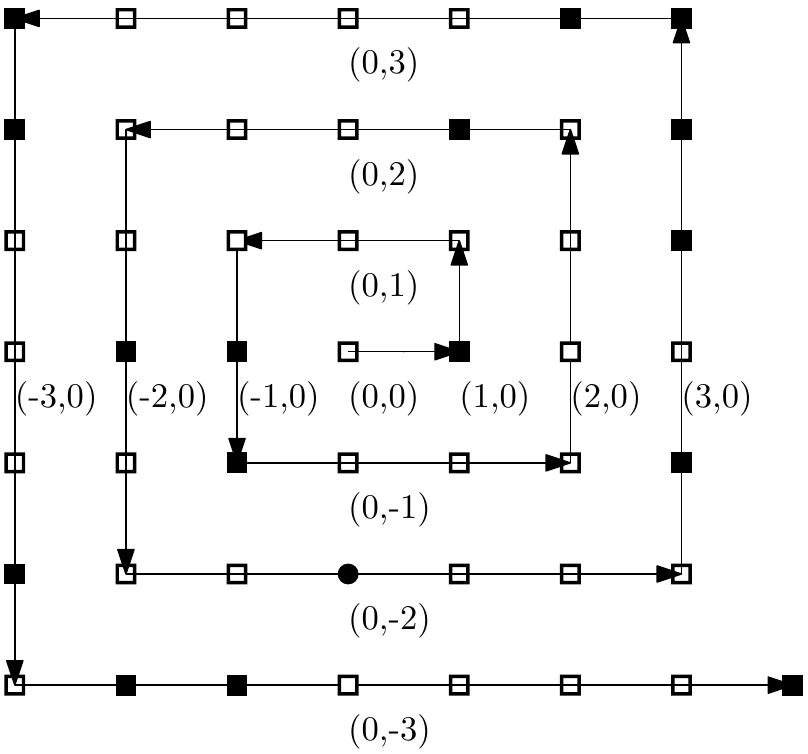}}
       		\vspace*{8pt}
       		\caption{A white box means that the value of a function 
       			$f: \mathbb{Z}^2 \rightarrow \mathbb{Z}_2$ is 
       			$e \in \mathbb{Z}_2$, 
       			a black box means that it is $h \in \mathbb{Z}_2$, 
       			a black disk at the point $(0,-2)$ specifies the position of the lamplighter  and tells us that the value of $f$ is $h$. 
       			For the element shown in this figure the word representing
       			it is $0100011000000100001000C_1000101111000011000101100001$. 
       		}
       		\label{z2wrz21fig}
       	\end{figure}     
       	Put $\Sigma = \{0,1,C_0,C_1\}$. 
       	Let us consider the map 
       	$t: \mathbb{N} \rightarrow \mathbb{Z} ^2$ such
       	that: $t(1) = (0,0)$, $t(2)=(1,0)$, $t(3) = (1,1)$, 
       	$t(4) = (0,1)$, $t(5)=(-1,1)$, $t(6)= (-1,0)$, $t(7)=(-1,-1)$,
       	$t(8)= (0,-1)$, and etc.; 
       	the map $t : \mathbb{N} \rightarrow \mathbb{Z} ^2$ is shown in Fig.~\ref{z2wrz21fig}.        
       	For a given element 
       	$(f,z) \in \mathbb{Z}_2 \wr \mathbb{Z}^2$,
       	 represent it as the word for which the $k$th 
       	symbol is $0$ if $f(t(k))=e$, $1$  if $f(t(k))=h$,
       	$C_0$ if $f(t(k))=e$ and $z = t(k)$, and 
       	$C_1$ if $f(t(k))=h$ and $z = t(k)$.
       	The last symbol of a word is not allowed to be $0$, i.e., 
       	it should be either $1$, $C_0$ or $C_1$.

       	It can be seen that the language $L \subset \Sigma^*$ of  representatives of all elements $g \in \mathbb{Z}_2 \wr \mathbb{Z}^2$
       	is regular. 
       	Also, the language $L_h = \{ w_1 \otimes w_2 | w_1, w_2 \in L, 
       	\psi (w_1) h  =  \psi (w_2) \}$
       	is regular.       
       	Let us consider the languages 
       	$L_x = \{ w_1 \otimes w_2 | w_1, w_2 \in L, 
       	\psi (w_1) x = \psi (w_2)  \} $ and 
       	$L_y = \{ w_1 \otimes w_2 | w_1, w_2 \in L, 
       	\psi (w_1) y = \psi (w_2) \}$.
       	We will show that $L_x$ and $L_y$ are indexed languages.         	       	       	            	       		       			      		
       		Put the stack alphabet  $\Xi = \{ I, B, T \}$. The symbols $B$ and $T$ denote 
       		the bottom and the top of the stack, respectively. 
       		The symbol $I$ is used for all intermediate positions.   
       		Let $w \in L_x$. Consider the stack automaton 
       		$\mathcal{M}_x$ that
       		works as follows until it meets for the first time the 
       		letter that contains $C_0$ or $C_1$.
       		
       		\begin{figure}[b]
       			\centerline
       			{\includegraphics[height=5.5cm]{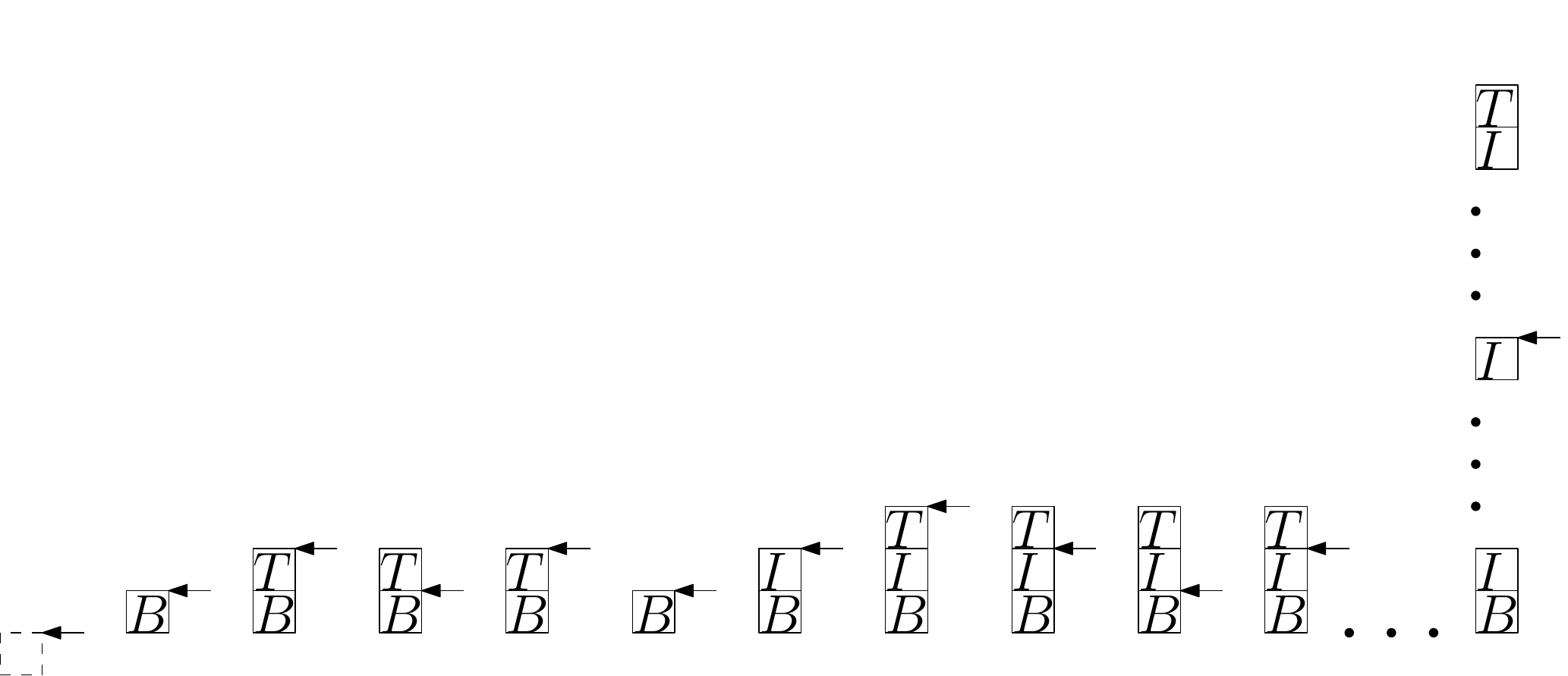}}
       			\vspace*{8pt}
       			\caption{The content of the stack is shown for the first  
       				several  iterations of $\mathcal{M}_x$. 
       				In general situation, the content of the stack is
       				shown to the right.}
       			\label{stackautomaton1} 
       		\end{figure} 	
       		
       		\begin{itemize}
       			\item{Initially the stack is empty.}  	
       			\item{$\mathcal{M}_x$ reads off the first letter of $w$ and
       				pushes  $B$ onto the stack.}	
       			\item{$\mathcal{M}_x$ reads off the second letter of $w$ and 
       				pushes  $T$ onto the stack.}
       			\item{$\mathcal{M}_x$ reads off the third letter of $w$ and 
       				goes one step down.} 
       			\item{$\mathcal{M}_x$ reads off the fourth letter of $w$ 
       				and goes one step up.}
       			\item{$\mathcal{M}_x$ makes two silent moves popping 
       				 $T$ and pushing $I$.}
       			\item{$\mathcal{M}_x$ reads off the fifth letter of $w$ 
       				and pushes $T$ onto the stack.}	
       			\item{$\mathcal{M}_x$ reads off the sixth letter of $w$ 
       				and goes one step down.}
       			\item{$\mathcal{M}_x$ reads off the seventh letter of $w$ 
       				and goes one step down.}
       			\item{$\mathcal{M}_x$ reads off the eights letter of $w$
       				and goes one step up.}
       			\item{The process continues by going up and down along the
       				stack between  $B$ and 
       				 $T$. Each time the top  is reached, it is raised up by one. 
       				This process is shown in Fig.~\ref{stackautomaton1}. 
       			     }              		 
       		\end{itemize}	
       	    
       		It is easy to see the following. 
       		If a letter being read of a word $w$ at a position $m$ contains $C_0$ or $C_1$ for the first time, then either the next letter at the position $m+1$ contains $C_0$ or $C_1$, or 
       		the letter at the position $m+ (4n+1)$ contains 
       		$C_0$ or $C_1$, where $n$ is the current height of the stack. In order to verify the latter case, the automaton 
       		$\mathcal{M}_x$, after meeting the letter containing  
       		$C_0$ or $C_1$ 
       		for the first time,
       		works as follows.
       		\begin{itemize}
       			\item{$\mathcal{M}_x$ makes silent moves going up until it reaches 
       				the top of the stack.} 
       			\item{Then the automaton $\mathcal{M}_x$ pops a symbol out of stack each time it reads off four consecutive letters.}
       			\item{After the stack is emptied the automaton reads off
       				  the next letter which should contain $C_0$ or $C_1$.}   
       			\item{After meeting the letter containing 
       			   	$C_0$ or $C_1$ for the second time, the automaton keeps working without using the stack.}
       		\end{itemize}       		
       		It is easy to see that the automaton $\mathcal{M}_x$ 
       		recognizes the language $L_x$. 
       		In  a similar way one can obtain the nested stack automaton 
       		$\mathcal{M}_y$ that recognizes the language $L_y$.    
       	
       	\end{proof}        	
  
       	
       	\begin{remark}      
       		For a given $g \in \mathbb{Z}_2 \wr \mathbb{Z}^2$,
       		let $w$ be the representative of $g$ for the 
       		indexed Cayley automatic representations of 
       		$\mathbb{Z}_2 \wr \mathbb{Z}^2$ constructed in  
       		Theorem \ref{Z2wrZ2indexautomaticprop}. Then the 
       		inequality $ 
       		|g| \leqslant 2 |w| - 1
       		$ holds.
       		However, no inequality of the form        	
       	    $\lambda |w| + \mu \leqslant |g|$    		
       		is satisfied for all $g \in \mathbb{Z}_2 \wr \mathbb{Z}^2$.  
       	\end{remark}


	\end{document}